\documentclass[10pt]{amsart}
\usepackage{ae,amsfonts,euscript,enumerate}
\usepackage{amsmath,euscript}
\usepackage{amssymb}
\usepackage{amsthm}
\usepackage{enumerate}
\usepackage{amsfonts}
\usepackage{comment}
\usepackage{colortbl}
\usepackage{a4wide}
\usepackage{amssymb,amsmath,array}

\newtheorem{thm}{Theorem}[section]
\newtheorem{lem}[thm]{Lemma}

\newtheorem{prop}[thm]{Proposition}
\newtheorem{cor}[thm]{Corollary}

\theoremstyle{defn}
\newtheorem{defn}{Definition}

\def\co{{\mathcal O}}

\def\oqmm13{\co_q(M_{1,3})}
\def\oqm23{\co_q(M_{2,3})}

\usepackage{amssymb}
\usepackage{amsmath}
\usepackage[latin1]{inputenc}

\usepackage{fullpage}
\usepackage{amsmath,euscript}
\usepackage{amssymb}
\usepackage{amsthm}
\usepackage{enumerate}
\usepackage{amsfonts}
\usepackage{comment}
\usepackage{colortbl}
\usepackage{a4wide}
\usepackage{amssymb,amsmath,array}

\title[]{Transcendence degree of division algebras}
\author{Jason P.~Bell}
\thanks{This work supported by NSERC grant 31-611456.}

\subjclass[2000]{}

\address{Jason Bell\\
Department of Mathematics\\
Simon Fraser University\\
Burnaby, BC V5A 1S6\\
Canada}

\email{jpb@math.sfu.ca}

\begin{document}
\bibliographystyle{plain}

\begin{abstract}  We define a transcendence degree for division algebras, by modifying the lower transcendence degree construction of Zhang.  We show that this invariant has many of the desirable properties one would expect a noncommutative analogue of the ordinary transcendence degree for fields to have.  Using this invariant, we prove the following conjecture of Small.  
Let $k$ be a field, let $A$ be a finitely generated $k$-algebra that is an Ore domain, and let $D$ denote the quotient division algebra of $A$.  If $A$ does not satisfy a polynomial identity then ${\rm GKdim}(K) \le {\rm GKdim}(A)-1$ for every commutative subalgebra $K$ of $D$.  
\end{abstract}
\maketitle

\section{Introduction}

Transcendence degree for fields is an important invariant, which has proven incredibly useful in algebraic geometry.  In the noncommutative setting, many different transcendence degrees have been proposed \cite{GK, BK, Re1, Sc1, St, Z, Z1, Z2}, many of which possess some of the desirable properties that one would hope for a noncommutative analogue of transcendence degree to possess.  Sadly, none of these has proved as versatile as the ordinary transcendence degree has in the commutative setting, as there has always been the fundamental problem: they are either difficult to compute in practice or are not powerful enough to say anything about division subalgebras.

The first such invariant was defined by Gelfand and Kirillov \cite{GK}, who used their Gelfand-Kirillov transcendence degree to prove that if the quotient division algebras of the $n$th and $m$th Weyl algebras were isomorphic, then $n=m$.   Gelfand-Kirillov transcendence degree is obtained from Gelfand-Kirillov dimension in a natrual way.  

Given a finitely generated algebra $A$ over a field $k$, the \emph{Gelfand-Kirillov dimension} (GK dimension, for short) of $A$ is defined to be
\[ {\rm GKdim}(A) \ = \ \limsup_{n\rightarrow\infty} \frac{\log\,{\rm dim} V^n}{\log \, n}, \]
where $V$ is a finite-dimensional $k$-vector subspace of $A$ which contains $1$ and generates $A$ as a $k$-algebra.  We note that this definition is independent of the choice of vector space $V$ with the above properties.  

The \emph{Gelfand-Kirillov transcendence degree} for a division algebra $D$ with centre $k$ is defined to be
\[ {\rm Tdeg}(A) \ = \ \sup_V \inf_{b} \limsup_{n\rightarrow\infty} \frac{\log\,{\rm dim}_k (k+bV)^n}{\log \, n}, \]
where $V$ ranges over all finite-dimensional $k$-vector subspaces of $D$ and $b$ ranges over all nonzero elements of $D$.

Zhang \cite{Z} introduced a combinatorial invariant, which he called the \emph{lower transcendence degree} of a division algebra $D$, which he denoted by ${\rm Ld}(D)$.  We define this degree in Section \ref{sec: def}.  Zhang showed that this degree had many of the basic properties that one would expect a transcendence degree to have.  In particular, he showed that if $k$ is a field, $A$ is a $k$-algebra that is a domain of finite GK dimension, and $D$ is the quotient division algebra of $A$, then
$${\rm GKdim}(A) \ \ge \ {\rm Ld}(D).$$
He also showed that ${\rm Ld}(K)={\rm trdeg_k}(K)$ in the case that $K$ is a field and that if $E$ is a division subalgebra of $D$ then ${\rm Ld}(E)\le {\rm Ld}(D)$.   We modify Zhang's construction to define a new transcendence degree, which we call the \emph{strong lower transcendence degree} and which we denote by ${\rm Ld}^*$.  We define this invariant in Section \ref{sec: def}.
We use the adjective strong, simply because we have the inequality
$${\rm Ld}^*(D)\ge {\rm Ld}(D).$$
We are able to show that the strong lower transcendence degree has the following properties.
\begin{enumerate}
\item If $D$ is a division algebra and $E$ is a division subalgebra of $D$ then ${\rm Ld}(E)\le {\rm Ld}^*(D)$.
\item If $D$ is a division subalgebra and $E$ is a division subalgebra of $D$ such that $D$ is finite-dimensional as either a left or right $E$-vector space, then ${\rm Ld}^*(D)\le {\rm Ld}^*(E)$.
\item If $D$ is a finitely generated division algebra and $E$ is a division subalgebra of $D$ such that $D$ is infinite-dimensional as a left $E$-vector space, then ${\rm Ld}^*(D)\ge {\rm Ld}(E)+1$.
\item If ${\rm Ld}^*(D)<1$ then ${\rm Ld}^*(D)=0$; moreover, ${\rm Ld}^*(D)=0$ if and only if every finitely generated subalgebra of $D$ is finite-dimensional.
\item If $k$ is a field and $A$ is a finitely generated $k$-algebra that is an Ore domain and $D$ is its quotient division algebra, then
$${\rm Ld}^*(D)\le {\rm GKdim}(A).$$
\item If $k$ is a field and $K$ is a field extension of $k$ of transcendence degree $d$ then ${\rm Ld}^*(K)=d$.
\end{enumerate}
Using these results, we prove the following conjecture of Small \cite[Conjecture 8.1]{Z}.

\begin{thm} Let $k$ be a field and let $A$ be a finitely generated $k$-algebra that is an Ore domain.  If $A$ does not satisfy a polynomial identity and $K$ is a commutative subalgebra of the quotient division algebra of $A$ then ${\rm GKdim}(K)\le {\rm GKdim}(A)-1$.\label{thm: main2}
\end{thm}

Zhang \cite[Corollary 0.8]{Z} showed that Theorem \ref{thm: main2} holds if the conclusion is replaced by ${\rm GKdim}(K)\le {\rm GKdim}(A)$, and, moreover, the hypothesis that $A$ not satisfy a polynomial identity is unnecessary with this bound.  

The outline of this paper is as follows.  In Section \ref{sec: def} we recall Zhang's definition of lower transcendence degree and define the strong lower transcendence degree.  In Section \ref{sec: prop}, we prove that the strong lower transcendence degree has properties (1), (2), (4), (5), and (6).  In Section \ref{sec: estimates}, we show that property (3) holds and we prove Theorem \ref{thm: main2}.

\section{Definitions}
\label{sec: def}
In this section, we recall the definition of lower transcendence degree, defined by Zhang \cite{Z} and recall some basic facts about this invariant.  We then proceed to modify his construction to provide a two-sided version of this invariant, which we call the \emph{strong lower transcendence degree} and which we denote by ${\rm Ld}^*$

Given a field $k$ and a $k$-algebra $A$, we say that a $k$-vector subspace $V$ of $A$ is a \emph{subframe} of $A$ if $V$ is finite-dimensional and contains $1$; we say that $V$ is a \emph{frame} if $V$ is a subframe and $V$ generates $A$ as a $k$-algebra.

The definition of lower transcendence degree is fairly technical and we refer the reader to Zhang \cite{Z} for more insight into this definition.   Let $k$ be a field and let $A$ be a $k$-algebra that is a domain.  If $V$ is a subframe of $A$, we define 
${\rm VDI}(V)$ to be the supremum over all nonnegative numbers $d$ such that there exists a positive constant $C$ such that
$${\rm dim}_k(VW)\ge {\rm dim}_k(W)+C{\rm dim}(W^{(d-1)/d})$$ for every subframe $W$ of $D$.  (If no nonnegative $d$ exists, we take ${\rm VDI}(V)$ to be zero.)  VDI stands for ``volume difference inequality'' and it gives a measure of the growth of an algebra.  We then define the \emph{lower transcendence degree} of $A$ by
$${\rm Ld}(A) = \sup_V {\rm VDI}(V),$$ where $V$ ranges over all subframes of $A$.  

The definition, while technical, gives a powerful invariant that Zhang \cite{Z} has used to answer many difficult problems about division algebras.  Zhang showed that if $A$ is an Ore domain of finite GK dimension and $D$ is the quotient division algebra of $A$ then ${\rm Ld}(A)\le {\rm GKdim}(A)$.  Moreover, equality holds for many classes of rings.  In particular, if $A$ is a commutative domain over a field $k$, then equality holds and so Lower transcendence degree agrees with ordinary transcendence degree.

One of the weaknesses of lower transcendence degree is that it is unknown whether it satisfies the equality ${\rm Ld}(D)={\rm Ld}(D^{{\rm op}})$.  To correct this, we use a two-sided approach.

We define 
${\rm VDI}^*(V)$ to be the supremum over all nonnegative numbers $d$ such that there exists a positive constant $C$ such that
$$\max\left({\rm dim}_k(VW),{\rm dim}_k(WV)\right)\ge {\rm dim}_k(W)+C{\rm dim}(W^{(d-1)/d})$$ for every subframe $W$ of $A$.  (As before, if no nonnegative $d$ exists, we take ${\rm VDI}^*(V)$ to be zero.)  We then define the strong lower transcendence degree of a domain $A$ by
$${\rm Ld}^*(A) = \sup_V {\rm VDI}^*(V),$$ where $V$ ranges over all subframes of $A$.
We note that we trivially have the estimate
\begin{equation}
{\rm Ld}^*(D)\ge \max({\rm Ld}(D),{\rm Ld}(D^{\rm op}))
\end{equation}
and by construction
$${\rm Ld}^*(D)={\rm Ld}^*(D^{\rm op}).$$

\section{Basic properties}
\label{sec: prop}
In this section, we prove the basic properties of strong lower transcendence degree.  For the most part, we follow the work of Zhang \cite{Z}.  We note that the first property listed in the introduction, namely that ${\rm Ld}(E)\le {\rm Ld}^*(D)$ whenever $E$ is a division subalgebra of $D$ follows immediately from the fact that ${\rm Ld}^*(D) \ge {\rm Ld}(D)$ and Theorem 2.4 of Zhang \cite{Z}.
 
We first show that if $k$ is a field, $A$ is an Ore domain that is a $k$-algebra, and $D$ is the quotient division algebra of $A$, then we have
\begin{equation}
{\rm Ld}(D)\le {\rm Ld}^*(D) \le {\rm GKdim}(A). 
\end{equation}
This is property (5) on the list of properties given in the introduction.
\begin{prop} Let $k$ be a field and let $A$ be a $k$-algebra that is a domain.  Then
$${\rm Ld}^*(A)\le {\rm GKdim}(A).$$
\end{prop}
\begin{proof} If ${\rm Ld}^*(D)=0$ or ${\rm GKdim}(A)=\infty$ there is nothing to prove.  Thus we assume that ${\rm Ld}^*(D)>0$ and ${\rm GKdim}(A)<\infty$.  Let $d$ be a positive number less than ${\rm Ld}^*(A)$.  

Then by assumption, there exists a subframe $V$ of $A$ and a positive constant $C$ such that $$\max({\rm dim}_k(VW), {\rm dim}_k(WV))\ge{\rm dim}_k(W)+C\left( {\rm dim}_k(W)\right)^{(d-1)/d}$$ for every subframe $W$ of $A$.  Letting $W=V^n$ gives
$$ {\rm dim}_k(V^{n+1})\ge {\rm dim}_k(V^n) +  C\left( {\rm dim}_k(V^n)\right)^{(d-1-\epsilon)/(d-\epsilon)}.$$
Telescoping gives,
\begin{eqnarray*}
{\rm dim}_k(V^{2n}) &\ge & {\rm dim}_k(V^{2n}) - {\rm dim}(V^n) \\
& \ge & C\sum_{j=n}^{2n-1} \left( {\rm dim}_k(V^j)\right)^{(d-1-\epsilon)/(d-\epsilon)} \\
&\ge & Cn \left( {\rm dim}_k(V^n)\right)^{(d-1-\epsilon)/(d-\epsilon)}.
\end{eqnarray*}
Let $e={\rm GKdim}(A)$ and let $\epsilon>0$.  Then there are infinitely many $n$ such that
${\rm dim}(V^n)\ge n^{e-\epsilon}$, but we must have
${\rm dim}(V^n)<n^{e+\epsilon}$ for all sufficiently large $n$.  In particular, there are infinitely many $n$ such that
$$(2n)^{e+\epsilon}>{\rm dim}_k(V^{2n}) \ge C n\cdot \left(n^{e-\epsilon}\right)^{(d-1)/d}.$$
Since this holds for infinitely many $n$, it follows that
$$e+\epsilon \ge 1 + (e-\epsilon)(d-1)/d$$ for every $\epsilon>0$.
Letting $\epsilon$ tend to zero gives
$$e\ge 1+e(d-1)/d$$ or 
equivalently, $e\ge d$.  The result follows.\end{proof}
This shows that lower transcendence degree does not blow up under localization.  Makar-Limanov \cite{ML} has shown that the quotient division algebra of the Weyl algebra over a field of characteristic $0$ contains a copy of the free algebra on two generators, and hence GK dimension generally blows up under localization, except when we are dealing with algebras that are in some sense very close to being commutative.  

As an immediate corollary, we obtain property (6).
\begin{cor} Let $k$ be a field and let $K$ be an extension of $k$.  Then ${\rm Ld}^*(K)={\rm GKdim}(K)$.
\end{cor}
\begin{proof} By a result of Zhang \cite[Corollary 2.8 (1)]{Z}, we have
$${\rm GKdim}(K) = {\rm Ld}(K)\le {\rm Ld}^*(K)\le {\rm GKdim}(K).$$
The result follows.
\end{proof}

We now show that the strong lower transcendence degree behaves as one would hope with respect to large division subalgebras.  The following proposition is a proof that property (2) holds.



\begin{prop} Let $D$ be a division algebra over a field $k$ and let $E$ be a division subalgebra.  If $D$ is finite-dimensional as both a left and right $E$-vector space then ${\rm Ld}^*(D)\le {\rm Ld}^*(E)$.
\end{prop}
\begin{proof} We modify the proof of Proposition 3.1 given by Zhang \cite{Z}.  We may assume that ${\rm Ld}^*(E)<\infty$, since otherwise there is nothing to prove in this case.  Thus we may assume that there is a nonnegative real number $d$ such that $d={\rm Ld}^*(E)$.  Let $\epsilon>0$.  
Write $D=x_1Ex_1\oplus \cdots \oplus x_pE=Ey_1\oplus \cdots \oplus Ey_q$ and  
pick a subframe $V$ of $D$. Then there exists a subframe $V_1$ of $E$ such that
$$Vx_1+\cdots +Vx_p \subseteq x_1V_1+\cdots+x_p V_1$$ and
$$y_1V+\cdots + y_qV\subseteq V_1y_1+\cdots + V_1 y_q.$$ 
Since ${\rm Ld}^*(E)=d$, there exists a subframe $W$ of $E$ such that
$\max\left({\rm dim}(WV_1/W),{\rm dim}(V_1W/W)\right)<\left({\rm dim}(W)\right)^{(d-1+\epsilon)/(d+\epsilon)}$.  Let $W_0=\sum_{i,j} x_iWy_j$.
Then
\begin{eqnarray*}
VW_0 & \subseteq & V\left( \sum_{i,j} x_iWy_j\right) \\
&\subseteq & \sum_{i,j} x_iV_1Wy_j\\
&=& \sum_{i,j} x_iWy_j + x_i Ty_j\\
&=& W_0+  \sum_{i,j}  x_i Ty_j,
\end{eqnarray*}
where $T$ is a finite-dimensional vector subspace which satisfies $V_1W = W\oplus T$.  Then
by assumption ${\rm dim}(T)<\left({\rm dim}(W_0)\right)^{(d-1+\epsilon)/(d+\epsilon)}$.
Hence $${\rm dim}\left( VW_0/W_0\right)\le pq \left({\rm dim}(W_0)\right)^{(d-1+\epsilon)/(d+\epsilon)}$$ and 
similarly,
$${\rm dim}\left( W_0V/W_0\right)\le pq \left({\rm dim}(W_0)\right)^{(d-1+\epsilon)/(d+\epsilon)}.$$
It follows that ${\rm Ld}^*(D)\le d$.    
\end{proof}
We next show that property (4) holds.
\begin{prop} Let $k$ be a field and let $D$ be a division algebra over $k$.  If ${\rm Ld}^*(D)<1$ then ${\rm Ld}^*(D)=0$; moreover, $Ld^*(D) = 0$ if and only if every finitely generated subalgebra 
of $D$ is finite-dimensional as a $k$-vector space. 
\end{prop}
\begin{proof}
Note that if ${\rm Ld}(D^*)=0$ then ${\rm Ld}(D)\le {\rm Ld}^*(D)=0$ and so every finitely generated subalgebra of $D$ is finite-dimensional over $k$ by Proposition 1.1 (4) of Zhang \cite{Z}.  Furthermore, if every finitely generated subalgebra of $D$ is finite-dimensional, then we necessarily have that ${\rm Ld}(D^*)=0$.  To see this, let $V$ be a subframe of $D$ and let $D_0$ be the finite-dimensional division subalgebra generated by $V$.  Then if $W$ is a subframe that is a left and right $D_0$-vector space, then $VW=WV=W$ and so ${\rm Ld}^*(D)=0$.  

On the other hand, if $D$ has a finitely generated division subalgebra that is not finite-dimensional over $k$, then ${\rm Ld}^*(D)\ge {\rm Ld}(D)\ge 1$ \cite[Prop 1.1 (2) \& (4)]{Z}.

\end{proof}
\section{Estimates}
\label{sec: estimates}
In this section, we prove the basic estimates that we will use to obtain a proof that property (3), given in the introduction, holds.  We will then use this to prove Theorem \ref{thm: main2}.
We introduce the notion of a decomposition of a vector space, which will be key in all of our estimates.
\begin{defn} {\em Let $k$ be a field, let $D$ be a division algebra over $k$, and let $W$ be a finite-dimensional $k$-vector subspace of $D$.  Given a division subalgebra $E$ of $D$ and a finite-dimensional $k$-vector subspace $V$ of $D$, we say that $W$ admits a left $(E,V)$-\emph{decomposition} if  there exist subspaces $U_1,\ldots ,U_r$ of $W$, $x_1,\ldots ,x_r\in V$, and natural numbers $a_1,\ldots ,a_r$ with $i<a_i\le r+1$ such that:

\begin{enumerate}

\item $W=U_1\oplus U_2\oplus \cdots \oplus U_r$;

\item $U_i x_i \subseteq EU_1+EU_2+\cdots +EU_{a_i}$, where $U_{r+1}=D$;

\item $U_i x_i \cap \left(EU_1+EU_2+\cdots +EU_{a_i-1}\right)=(0)$;

\item $U_j x_i \subseteq EU_1+EU_2+\cdots +EU_{a_i-1}$ for $j<i$.

\end{enumerate}
In this case, we will write $U_1\oplus \cdots \oplus U_r$ is a left $(E,V)$-decomposition of $W$.}
\label{def: 14}  The notion of a right $(E,V)$-decomposition is defined analogously.
\end{defn}

We show that under general conditions such decompositions exist.
\begin{lem} Let $k$ be a field, let $D$ be a division algebra over $k$, and let $E$ be a division subalgebra of $D$.  If $W$ and $V$ are non-trivial subframes of $D$ such that $WV\not\subseteq EW$, then $W$ admits a left $(E,V)$-decomposition.  
\label{lem: decomp}
\end{lem}

\begin{proof}  We prove this by induction on the dimension of $W$.

 If the dimension of $W$ is $1$, then there exists some $x_1\in V$ such that $Wx_1\cap EW=(0)$; otherwise, $WV\subseteq EW$, a contradiction.  We then take $U_1=W$ and $a_1=2$ and obtain the result in this case.

  We next assume that the conclusion of the statement of the proposition holds for all $k$-vector subspaces of $D$ whose dimension is strictly less than the dimension of $W$.

Since $WV\not\subseteq EW$, there exists $x\in V$ such that $Wx\not\subseteq EW$.  Let $$W_1=\{w\in W~: ~wx\in EW\}.$$
Pick a subspace $W_0$ of $W$ such that $$W_0\oplus W_1=W.$$  By the inductive hypothesis, $W_1$ has a left $(E,V)$-decomposition $$W_1=U_1\oplus \cdots \oplus U_r.$$
Furthermore, there exist $x_1,\ldots ,x_r\in V$ and natural numbers $a_1,\ldots ,a_r$ such that the conditions (1)--(4) of Definition \ref{def: 14} are satisfied.

Let \begin{equation} S= \{i : 1\le i\le r, ~a_i=r+1\} \qquad {\rm and}\qquad T = \{1,2,\ldots ,r+1\}\setminus S.
\end{equation}
 For $i\in S$, it is possible that $U_i x_i \cap (EU_1+\cdots + EU_r+EW_0)\not = (0)$.  Thus we let $$U_{i,0}=\{u\in U_i~:~ux_i\in EU_1+\cdots + EU_r+EW_0$$ and choose $U_{i,1}$ such that $$U_{i,0}\oplus U_{i,1}\  = \ U_i.$$  For $i\in S$, we let $x_{i,0}=x_{i,1}=x_i$ and $a_{i,0}=r+1$, $a_{i,1}=r+2$; and we let $U_{r+1}=W_0$, $x_{r+1}=x$, and $a_{r+1}=r+2$.

We construct a left $(E,V)$-decomposition of $W$ using the subspaces $U_j$ with $j\in T$ and $U_{i,0}, U_{i,1}$ with $i\in S$.  We create a total ordering on the indices by declaring \begin{equation}
(i-1,1)\ < \ (i,0) \ < \ i \ < \ (i,1) \ < \ (i+1,0)
\end{equation}
for every natural number $i$.  

Notice that for $i\in S$, we have $EU_i = EU_{i,0}+EU_{i,1}$.  
Then for $j\in T$ with $j<r+1$ we have:
\begin{equation} U_j x_j \in EU_1+\cdots + EU_{a_j} = \sum_{i\in  T, i\le a_j} EU_i + \sum_{i\in S, i\le a_j}\left( EU_{i,0}+EU_{i,1}\right);\end{equation}

\begin{equation} U_j x_j \cap \left( \sum_{i\in  T, i\le a_j} EU_i + \sum_{i\in S, i\le a_j}\left( EU_{i,0}+EU_{i,1}\right)\right)=(0).\end{equation}

For $j\in S$ we have:

\begin{equation}U_{j,0} x_{j,0} \in \sum_{i\in  T} EU_i + \sum_{i\in S}\left( EU_{i,0}+EU_{i,1}\right)+ EU_{r+1};\end{equation}

\begin{equation} U_{j,0}x_{j,0}\cap \left( \sum_{i\in  T} EU_i + \sum_{i\in S}\left( EU_{i,0}+EU_{i,1}\right)\right)=(0);\end{equation}

\begin{equation} U_{j,1}x_{j,1}\cap \left( \sum_{i\in  T} EU_i + \sum_{i\in S}\left( EU_{i,0}+EU_{i,1}\right)+EU_{r+1}\right)
=(0).\end{equation}
Finally, we take $x_{r+1}=x$.  Then by construction, $U_{r+1}x_{r+1}=W_0x$ which has trivial intersection with $EU_1+\cdots +EU_r$.  Thus these subspaces give a left $(E,V)$-decomposition of $W$. \end{proof}
A similar result holds for right decompositions.

\begin{lem}  Let $k$ be a field, let $D$ be a division algebra over $k$, and let $W$ and $V$ be non-trivial subframes of $D$.  Suppose that $E$ is a division subalgebra and $U_1\oplus \cdots \oplus U_r$ is a left $(E,V)$-decomposition of $W$.  Then
$EU_1+\cdots + EU_r$ is direct. \label{direct}
\end{lem}
\begin{proof} Suppose not.  Then for $1\le i\le r$ we can find $w_i\in EU_i$ with $w_1,\ldots ,w_r$ not all zero such that 
$$w_1+\cdots + w_r=0.$$
Then there exist $x_1,\ldots ,x_r\in V$, and natural numbers $a_1,\ldots ,a_r$ with $i<a_i\le r+1$ satisfying conditions (1)--(4) of Definition \ref{def: 14}.

Let $m$ be such that $w_m$ is nonzero and $w_i=0$ for every $i>m$.
Note that $w_i x_m\subseteq EU_1+\cdots + EU_{a_m-1}$ for $i<m$, and so
$$w_m x_m = (w_1+\cdots +w_{m-1})x_m \subseteq  EU_1+\cdots + EU_{a_m-1}.$$
But by hypothesis, $w_m x_m \cap ( EU_1+\cdots + EU_{a_m-1})=(0)$, and so $w_m x_m=0$.  Since $D$ is a domain, $w_m=0$, a contradiction.  Thus we obtain the desired result. 
\end{proof}

We now give two estimates which we will use to estimate the strong lower transcendence degree of division subalgebras.  This first lemma is rather technical and is where we really use all requirements listed in the definition of left $(E,V)$-decompositions.

\begin{lem}  Let $k$ be a field, let $D$ be a division algebra over $k$ and let $E$ be a division subalgebra of $D$.  Suppose that $W$ and $V$ are non-trivial subframes of $D$ and that $W$ admits a left $(E,V)$-decomposition $U_1\oplus \cdots \oplus U_r$.
Then
\[ {\rm dim}_k(W+WV) \ \ge \ {\rm dim}_k(W) + \max_{1\le i\le r} {\rm dim}_k\left(U_i\right).\]
\label{lem: Z1}
\end{lem}    
\begin{proof}
By assumption, there exist $x_1,\ldots, x_r$ in $V$ and natural numbers $a_1,\ldots ,a_r$ satisfying conditions (1)--(4) of Definition \ref{def: 14}.
 Pick $i_{0}$ such that 
$${\rm dim}_k(U_{i_0}) \ \ge \ {\rm dim}_k(U_j)$$ for $1\le j\le r$.
Let $i_1=a_{i_0}$.  If $i_1>r$, then we let $Y_0=U_{i_0}$; otherwise, by assumption there is a $k$-vector space embedding of $U_{i_0}x_{i_0}$ in $$\left(EU_1+\cdots+EU_{i_1}\right)/\left(EU_1+\cdots +EU_{i_1-1}\right).$$

By Lemma \ref{direct}, $U_{i_1}$ also embeds into $$\left(EU_1+\cdots+EU_{i_1}\right)/\left(EU_1+\cdots +EU_{i_1-1}\right),$$ and so it follows that there exists a subspace $Y_0$ of $U_{i_0}x_{i_0}$ such that the image of $Y_0$ in $$\left(EU_1+\cdots+EU_{i_1}\right)/\left(EU_1+\cdots +EU_{i_1-1}\right)$$
intersects the image of $U_{i_1}$ trivially and their sum is the image of $U_{i_0}x_{i_0}+U_{i_1}$.
Then 
\begin{equation}
{\rm dim}(Y_0) \ \ge \ {\rm dim}(U_{i_0})-{\rm dim}(U_{i_1})
\end{equation} if $i_1\le r$; otherwise, ${\rm dim}(Y_0)={\rm dim}(U_{i_0})$.

If $i_1>r$, then we stop; otherwise, we can repeat the procedure, taking $i_2=a_{i_1}$, and we can construct a subspace $Y_1$ of $U_{i_1}x_{i_1}$.  If $i_2>r$, then $Y_1=U_{i_1}x_{i_1}$; otherwise, we take $Y_1$ such that its image in $$\left(EU_1+\cdots+EU_{i_2}\right)/\left(EU_1+\cdots +EU_{i_2-1}\right)$$ has trivial intersection with the image of $U_{i_2}$ and its sum with the image of $U_{i_2}$ is the image of $U_{i_1}x_{i_1}+U_{i_2}$.

Then 
\begin{equation}
{\rm dim}(Y_1) \ \ge \ {\rm dim}(U_{i_1})-{\rm dim}(U_{i_2})
\end{equation}
if $i_2\le r$; otherwise, ${\rm dim}(Y_1)={\rm dim}(U_{i_1})$.

If we continue in this manner, we eventually reach an index $\ell$ such that $i_{\ell+1}=r+1$.
Notice $$WV+W \supseteq Y_0+\cdots +Y_{\ell}+W.$$ Moreover, we claim that the sum on the right is direct.  If not, there exists a dependence 
$$y_0+y_1+\cdots +y_{\ell}+u_1+\cdots + u_r=0,$$ with $y_i\in Y_i$, $u_j\in U_j$ not all zero.
Let $j$ be the largest index with $y_j\not = 0$.  Then $y_j\in EU_1+\cdots +EU_{i_{j+1}}$.  By Lemma \ref{direct}, $EU_1+\cdots +EU_r$ is direct, and so $u_n=0$ for $n>i_{j+1}$.  Then $y_j+u_{i_{j+1}}\in EU_1+\cdots + EU_{i_{j+1}-1}$, where we take $u_{r+1}=0$.  Thus the image of $y_j+u_{i_{j+1}}$ in $$\left(EU_1+\cdots+EU_{i_{j+1}}\right)/\left(EU_1+\cdots +EU_{i_{j+1}-1}\right)$$
is trivial, and so $y_j=0$ by construction of the space $Y_j$.  This contradicts the fact that the $y_i$ cannot all be zero.  Thus we see that the sum
$$Y_0+\cdots +Y_{\ell}+W$$ is direct.

Hence 
\[ {\rm dim}(WV+V) \ \ge \ {\rm dim}(W) + \sum_{i=1}^{\ell} {\rm dim}(Y_i).\]
At this point, we use telescoping sums:
\begin{eqnarray*}
\sum_{i=0}^{\ell} {\rm dim}(Y_i) &=& \sum_{j=0}^{\ell-1} \left(  {\rm dim}(U_{i_j})-{\rm dim}(U_{i_{j+1}}) \right)
+ {\rm dim}(U_{i_{\ell}}) \\
&=& {\rm dim}(U_{i_0}) \\
&=&  \max_{1\le i\le r} {\rm dim}_k\left(U_i\right).
\end{eqnarray*}
The result now follows. \end{proof}

\begin{lem} 
 Let $k$ be a field, let $D$ be a division algebra over $k$, let $E$ be a division subalgebra of $D$ of lower transcendence degree $d$, and let $V$ be a subframe of $D$.  For every $\epsilon>0$ there exists a subframe $V'\supseteq V$ and a positive constant $C>0$ such that whenever $U_1\oplus \cdots \oplus U_r$ is a left $(E,V')$-decomposition of a finite-dimensional $k$-vector subspace $W$ of $D$, we have
\[{\rm dim}_k(V'W)  \ \ge \ {\rm dim}_k(W) + \sum_{i=1}^r C \left({\rm dim}_k\left(U_i\right)\right)^{\frac{d-1-\epsilon}{d-\epsilon}}.\] 
\label{lem: Z2}
\end{lem}
\begin{proof} Let $\epsilon>0$.  
By definition of lower transcendence degree, there is some subframe $V_0$ of $E$ and a positive constant $C$ such that
\[  {\rm dim}_k\left(V_0U\right) \ \ge \  
{\rm dim}_k\left(U\right) + C\left({\rm dim}_k(U)\right)^{(d-1-\epsilon)/(d-\epsilon)} \] for every subframe $U$ of $E$.  

We let $V'=V+V_0$ and let $W$ be a subframe of $D$.  Suppose that $U_1\oplus \cdots \oplus U_r$ is a left $(E,V')$-decomposition of $W$ and let
\[b_i \ :=  \ {\rm dim}_k(U_i).  \] Then
\[  {\rm dim}_k\left(V'U_i\right) \ \ge \  {\rm dim}_k\left(V_0U_i\right) \ \ge \
{\rm dim}_k\left(U_i\right) + Cb_i^{(d-1-\epsilon)/(d-\epsilon)} \] for all $i $. 

By Lemma \ref{direct}, the sum
$$EU_1+\cdots +EU_r$$ is direct and since $V_0\subseteq E$, we see
\begin{eqnarray*}
{\rm dim}_k(V'W) - {\rm dim}_k(W) &\ge &
 {\rm dim}_k(V_0W) - {\rm dim}_k(W) \\ & = &
 \sum_{i=1}^r  {\rm dim}_k \left(V_0U_i/U_{i}\right) \\
 &\ge &\sum_{i=1}^r C b_i^{(d-1-\epsilon)/(d-\epsilon)}.
\end{eqnarray*}
\end{proof}

We now give a simple estimate which will allow us to combine the preceding two estimates.

\begin{lem} Let $b_1,\ldots ,b_m, d$ be positive real numbers and let $N$.  If
$$b_1+\cdots + b_r = N,$$ then either:
\begin{enumerate}
\item{$b_i\ge \left( \frac{d-1}{d}\right)^d N^{d/(d+1)}$ for some $i$; or}
\item{$b_1^{(d-1)/d}+\cdots + b_r^{(d-1)/d}\ge N^{d/(d+1)}$.}
\end{enumerate}
\label{lem: Z3}

\end{lem}
\begin{proof} Without loss of generality, we may assume that
$$b_1 \ \ge \ b_2 \ \ge \ \cdots \ \ge \ b_m.$$
Suppose that
$$b_1\le \frac{(d-1)^d}{d^d} N^{d/(d+1)}.$$
By the mean value theorem
$$b_i^{(d-1)/d}-b_{i+1}^{(d-1)/d} \ge (b_i-b_{i+1})\frac{d-1}{d} b_i^{-1/d} \ge
(b_i-b_{i+1})\frac{d-1}{d} \frac{d}{d-1} N^{-1/(d+1)}.$$
Thus
\begin{eqnarray*}
\sum_{i=1}^m b_i^{(d-1)/d} &=& 
m b_m^{(d-1)/d}+ \sum_{i=1}^{m-1} i (b_i^{(d-1)/d}-b_{i+1}^{(d-1)/d})\\
&\ge & m b_m^{(d-1)/d} + \sum_{i=1}^{m-1} i (b_i-b_{i+1})\frac{d-1}{d} \frac{d}{d-1} N^{-1/(d+1)} \\
&=& m b_m^{(d-1)/d} +  N^{-1/(d+1)} (b_1+\cdots + b_m - mb_m) \\
&=& mb_m^{(d-1)/d} +  N^{d/(d+1)} - mb_m N^{-1/(d+1)} \\
&=&  N^{d/(d+1)} + mb_m^{(d-1)/d}\Big(1 - \big(b_mN^{-d/(d+1)}\big)^{1/d}\Big) \\
&\ge & N^{d/(d+1)}.
\end{eqnarray*}
The result follows.  \end{proof}
We now prove property (3) in the list of properties given in the introduction.
\begin{thm} Let $k$ be a field and let $D$ be a finitely generated division algebra over $k$.  If $E$ is a division subalgebra of $D$ with the property that $D$ is infinite-dimensional as a left $E$-vector space, then $${\rm Ld}^*(D)\ge {\rm Ld}(E)+1.$$
\label{thm: mainx}
\end{thm}
\begin{proof} If ${\rm Ld}(E)=\infty$, there is nothing to prove, as ${\rm Ld}^*(D)\ge {\rm Ld}(D)\ge {\rm Ld}(E)=\infty$.  Thus we may assume that there is a positive real number $d$ such that ${\rm Ld}(E)=d$.  Since $D$ is finitely generated and is infinite dimensional as a left $E$-vector space, we may pick a subframe $V$ of $D$ such that
$EV^{n+1}$ properly contains $EV^n$ for every natural number $n$.

Let $W$ be a subframe of $D$.  We note that
$WV\not \subseteq EW$; otherwise, we would have
$V^n\subseteq WV^n\subseteq EW$ for every natural number $n$ and so $EV^n\subseteq EW$ for every natural number $n$, and so there must exist some $n$ such that $EV^n=EV^{n+1}$, a contradiction.   Thus $W$ admits a left $(E,V)$-decomposition by Lemma \ref{lem: decomp}.  Similarly, $W$ must admit a left $(E,V')$-decomposition for every subframe $V'$ containing $V$.

Let $\epsilon>0$.  Then by Lemma \ref{lem: Z2}, there exists a frame $V'\supset V$ and a positive constant $C>0$ such that
if $U_1\oplus \cdots \oplus U_r$ is a left $(E,V')$-decomposition of $W$ then
\[({\rm dim}_k(W+V'W) \ \ge \  {\rm dim}_k(W) + C\sum_{i=1}^{ r} {\rm dim}_k\left(U_i\right)^{(d-1-\epsilon)/(d-\epsilon)}.\]
Similarly,
By Lemma \ref{lem: Z1} 
\[ {\rm dim}_k(W+WV') \ \ge \ {\rm dim}_k(W) + \max_{1\le i\le r} {\rm dim}_k\left(U_i\right)\]
Let $b_i={\rm dim}_k\left(U_i\right)$ for $1\le i\le r$ then we have
$b_1+\cdots + b_r={\rm dim}(W)$.
By Lemma \ref{lem: Z3}, there is a constant $C_0>0$, independent of $W$, such that
$$\max\left( {\rm dim}_k(W+WV'),  {\rm dim}_k(W+V'W)\right) \ge {\rm dim}_k(W) + C_0\left({\rm dim}_k(W)\right)^{(d-\epsilon)/(d+1-\epsilon)}$$ for every subframe $W$ of $D$.
Thus by definition, ${\rm Ld}^*(D)\ge d+1-\epsilon$.  Since this holds for every $\epsilon>0$, we obtain the desired result.
\end{proof}
As an immediate corollary, we obtain the proof of Theorem \ref{thm: main2}

\begin{proof}[Proof of Theorem \ref{thm: main2}]  
We may assume that $A$ has finite GK dimension.

Let $D$ denote the quotient division algebra of $A$ and let $K$ be a subfield of $D$ that contains $k$.  If ${\rm GKdim}(K)>{\rm GKdim}(A)-1$ then we have
$${\rm Ld}(K)= {\rm GKdim}(K)>{\rm GKdim}(A)-1\ge {\rm Ld}^*(D)-1.$$ By Theorem \ref{thm: mainx} we have that $D$ must be finite-dimensional as a left $K$-vector space and hence $D$ embeds in a matrix ring over a field.  But this gives that $A$ satisfies a polynomial identity, a contradiction.  The result follows.
\end{proof}
\section{Concluding remarks and questions}
We make a few remarks.  Ideally, a transcendence degree should have the property that if $D$ is a finitely generated division algebra and $E$ is a division subalgebra such that $D$ is infinite-dimensional as a left $E$-vector space, then the transcendence degree of $E$ should be at most the transcendence degree of $D$ minus $1$.  We ask if this property holds for the strong lower transcendence degree.  This would have profound implications.  In particular, it would show that if 
$$k=D_0\subseteq D_1 \subseteq D_2 \subseteq \cdots \subseteq D_n=D$$ is a chain of finitely generated division subalgebras of $D$ such that each $D_i$ is infinite-dimensional as a left $D_{i-1}$-vector space. Then $n\le {\rm Ld}^*(D)$.  This is Zhang's conjecture \cite[Conjecture 8.4]{Z}.
The author \cite{Bell1} proved this in the case that $D$ is the quotient division algebra of a domain of GK dimension strictly less than $3$.
This is related to Schofield's notion of stratiform length \cite{Sc}.

Schofield has pathological constructions of division algebras $D$ which are finite-dimensional over a division subalgebra on one side but are infinite-dimensional on the other \cite[Section 5.9]{Co}.  In the case that we are dealing with division algebras of finite transcendence degree, however, it is expected that these type of phenomena should not occur.  Again, an inequality of this sort could be used that division algebras of finite transcendence degree are well-behaved in this sense.
\section*{Acknowledgments} The author thanks Lance Small, Dan Rogalski, and James Zhang for many helpful comments and suggestions.


\begin{thebibliography}{99}

\bibitem{ArtSt}
M. Artin and J. T. Stafford. Noncommutative graded domains with quadratic growth.  \emph{Invent. 
Math.} {\bf 122} (1995), 231--276.


\bibitem{Bell1} J. P. Bell. Division algebras of Gelfand-Kirillov transcendence degree 2. \emph{Israel J. Math.} {\bf 171} (2009), 51--60.
\bibitem{BK} W. Borho and H. Kraft, \"Uber die Gelfand-Kirillov dimension. \emph{Math. Ann.} {\bf 220} (1976), 1--24.
\bibitem{Co} P. M. Cohn. \emph{Skew fields: theory of general division rings.}
 Cambridge Univ.
Press, Cambridge, USA, 1995.

\bibitem{GK} I. M. Gelfand and A. A. Kirillov. Sur les corps li\' es aux alg\` ebres enveloppantes des
alg\`ebres de Lie. \emph{Publ. Math. IHES} {\bf 31} (1966), 5--19.

\bibitem{KL} G. Krause and T. Lenagan.  \emph{Growth of Algebras and Gelfand-Kirillov Dimension}, revised
edition.  Graduate Studies in Mathematics, no. 22. American Mathematical Society, Providence, 2000.
\bibitem{ML}
L. Makar-Limanov, The skew field $D_1$ contains free algebras. \emph{Comm. Algebra} {\bf 11}
(1983), 2003--2006.
\bibitem{Re1} R. Resco. Dimension theory for division rings. \emph{Israel J. Math.} {\bf 35} (1980), 215--221.
\bibitem{Sc1} A. H. Schofield. Questions on skew fields, in \emph{Methods in ring theory} (F. van
Ostaeyen, Ed.), pp. 489--495, Reidel, Dordrecht, 1984.
\bibitem{Sc} A. H. Schofield. Stratiform simple artinian rings. \emph{Proc. London Math. Soc.} {\bf 53} (1986), 267--287.


\bibitem{Sm}
A. Smoktunowicz.  The Artin-Stafford gap theorem.  \emph{Proc. Amer. Math. Soc.} {\bf 133} No. 7 (2005), 1925--1928.







\bibitem{Sm2} A. Smoktunowicz.  There are no graded domains with GK dimension 
strictly between $2$ and $3$.
\emph{Invent. Math.} {\bf 164} (2006), 635--640.
\bibitem{St} J. T. Stafford. Dimensions of division rings. \emph{Israel J. Math.} {\bf 45} (1983), 33--40.
\bibitem{Z} J. J. Zhang.  On Lower Transcendence Degree.  \emph{Adv. Math.} {\bf 139} (1998), 157--193.
\bibitem{Z1} J. J. Zhang.
On Gelfand-Kirillov transcendence degree. 
\emph{Trans. Amer. Math. Soc.} {\bf 348} (1996), no. 7, 2867--2899. 
\bibitem{Z2} A. Yekutieli and J. J. Zhang. Homological transcendence degree. \emph{Proc. London Math. Soc.} (3) {\bf 93} (2006), no. 1, 105--137.
\end{thebibliography}
\end{document}